\documentclass[12pt]{article}

\usepackage{graphicx, amssymb, latexsym, amsfonts, amsmath, lscape, amscd,
amsthm, color, epsfig, mathrsfs, tikz, enumerate, xcolor}
\usepackage[normalem]{ulem} 

\newtheorem{theorem}{Theorem}

\newcommand\DELETE[1]{}

\newcommand\Z{\mathbb{Z}}
\newcommand\supp{\operatorname{supp}}

\begin{document}


\title{{\bf A short proof of Seymour's 6-flow theorem}}
\author{
   Matt DeVos\thanks{
     Email: {\tt mdevos@sfu.ca}. 
     Supported by an NSERC Discovery Grant (Canada)}
 \and
   Kathryn Nurse\thanks{
  Email: {\tt knurse@sfu.ca} 
}
}


\maketitle

\begin{abstract}
We give a compact variation of Seymour's proof that every $2$-edge-connected graph has a nowhere-zero $\Z_2 \times \Z_3$-flow.
\end{abstract}


All graphs are finite; loops and multiple edges are allowed. For notation not defined here we use~\cite{DiestelReinhard2017GTbR}. Let $G=(V,E)$ be a directed graph, $A$ an additively-written abelian group, and $f: E \to A$ a function. We say $f$ is an $A$-flow whenever $\sum_{e \in \delta^+(v)}f(e) = \sum_{e \in \delta^-(v)}f(e)$ holds for every $v \in V$, where $\delta^+(v)$ ($\delta^-(v)$) is the set of edges whose initial (terminal) vertex is $v$. If $0 \not\in f(E)$, then we say $f$ is \emph{nowhere-zero}. If $f$ is a $\Z$-flow with $f(E) \subseteq \{0, \pm1, \pm2, \dots, \pm(k-1)\}$, then we say $f$ is a \emph{$k$-flow}. Note that reversing an edge $e$ and replacing $f(e)$ with its negation preserves all of the aforementioned properties; accordingly the existence of a nowhere-zero $A$-flow or $k$-flow depends only on the underlying graph. A famous conjecture of Tutte \cite{Tutte} asserts that every $2$-edge-connected graph has a nowhere-zero $5$-flow.  This conjecture remains open with the best result due to Seymour \cite{SeymourP.D1981N6} who proved that such graphs have nowhere-zero $6$-flows.  His argument involves a standard reduction due to Tutte equating the existence of a nowhere-zero $k$-flow and a nowhere-zero $A$-flow whenever $|A|=k$, together with the following central result. 

\begin{theorem}[Seymour]
    Every $2$-edge-connected digraph has a nowhere-zero $\Z_2 \times \Z_3$-flow.
\end{theorem}

We give a compact version of Seymour's proof. We prove a slightly stronger statement by induction, using simple contraction-based arguments
(and in particular, we don't require any special reductions). 
Our proof of Theorem \ref{main} relies on contracting a set of edges $S$, finding a flow $f$ in the smaller graph, then uncontracting $S$ and extending the domain of $f$ to include $S$. Observe that 1. it is always possible to extend the domain while maintaining that $f$ is a flow, and 2. if $S$ is a set of at least two parallel edges, and the abelian group has size at least 3, then it is possible to extend the domain so that additionally $0 \not\in f(S)$. In the following we use $G/S$ to denote the graph obtained from $G$ by contracting the set of edges $S \subseteq E$, and $\delta(u)$ is the set of edges incident to vertex $u$.


\begin{theorem}\label{main}
    If $G = (V,E)$ is a $2$-edge-connected digraph, and $u \in V$, then $G$ has a nowhere-zero flow $f_2 \times f_3 : E \to \Z_2 \times \Z_3$ so that $\delta(u) \cap \supp(f_2) = \emptyset$.
\end{theorem}

\begin{proof}
We proceed by induction on $|V|$, with the base case $|V| = 1$ holding trivially.
    First, suppose $G-u$ has a 1-edge-cut $E(V_1,V_2) = \{e\}$. Choose a partition $\{E_1,E_2\}$ of $E\setminus \{e\}$ so that for $i = 1,2$ the edges in $E_i$ have ends in $V_i \cup \{u\}$. Let $G_i = G/E_i$. By induction, $G_i$ has a nowhere-zero flow $f^i_2 \times f^i_3 : E(G_i) \to \Z_2 \times \Z_3$ so the edges incident to the contracted vertex are not in the support of $f_2^i$. By possibly replacing $f^1_3$ with its negation, we may assume that $f^1_3 (e) = f^2_3(e)$ and then the $\Z_2 \times \Z_3$-flows in each $G_i$ combine to give the desired flow in $G$.  Thus we may assume $G-u$ has no cut-edge.

    Choose distinct edges $ux$ and $ux'$ so that $x$ and $x'$ are in the same component of $G-u$ (possibly $x = x'$).  By our assumptions, we may choose two edge-disjoint paths $P_1, P_2 \subseteq G-u$ from $x$ to $x'$.
    Set $H = P \cup P'$, $S = E(u,V(H))$, $G_1 = G/E(H)$ with contracted vertex $u_1$, and $G_2 = G_1/ S$ with contracted vertex $u_2$. By induction, $G_2$ has a flow $f_2 \times f_3 : E(G_2) \to \Z_2 \times \Z_3$ so that $\delta(u_2) \cap \supp(f_2) = \emptyset$. Because $S$ is a set of at least two parallel edges, we may extend $f_3$ to $E(G_1)$ so that it remains a flow and $f_3(e) \neq 0$ for all $e \in S$. Because $\delta(u_2) \cap \supp(f_2) = \emptyset$, setting $f_2(e) = 0$ for all $e \in S$ extends $f_2$ to $E(G_1)$ keeping it a flow. Note that $\delta_{G_1}(u) \cap \supp(f_2) = \emptyset = \delta_{G_1}(u_1) \cap \supp(f_2)$. Now, further extend $f_3$ to $E(G)$ so that it remains a flow. Because $\delta(u_1) \cap \supp(f_2) = \emptyset$, and every vertex of $H$ has even degree, we may extend $f_2$ to $E(G)$ by setting $f_2(e) = 1$ for all $e \in E(H)$, keeping it a flow. Now $S \subseteq \supp(f_3), E(H) \subseteq \supp(f_2)$ and so $f_2 \times f_3$ is as desired.
\end{proof}

 \bibliographystyle{abbrv}
\bibliography{bib}

\begin{thebibliography}{1}

\bibitem{DiestelReinhard2017GTbR}
R.~Diestel.
\newblock {\em Graph Theory}, volume 173 of {\em Graduate Texts in
  Mathematics}.
\newblock Springer Berlin Heidelberg : Imprint: Springer, 2017.

\bibitem{SeymourP.D1981N6}
P.~Seymour.
\newblock Nowhere-zero 6-flows.
\newblock {\em Journal of Combinatorial Theory. Series B}, 30(2):130--135,
  1981.

\bibitem{Tutte}
W.~T. Tutte.
\newblock A contribution to the theory of chromatic polynomials.
\newblock {\em Canad. J. Math}, pages 80--91, 1954.

\end{thebibliography}

\end{document}